\documentclass{amsart}
\usepackage{amssymb}
\usepackage{xcolor}
\title{On a Conjecture of Erd\H os, Gallai, and Tuza}
\author{Gregory J.~Puleo}
\renewcommand{\subset}{\subseteq}
\newcommand{\join}{\vee}
\newcommand{\iso}{\cong}
\newcommand{\lfrac}[2]{#1/#2}

\newcommand{\bad}[1]{\textcolor{red}{\emph{#1}}}
\newcommand{\citethis}[1]{\bad{[CITE THIS]}}

\newcommand{\sizeof}[1]{\left\lvert{#1}\right\rvert}

\newcommand{\aph}{\alpha_1}
\newcommand{\hang}[1]{\makebox[0cm]{#1}}
\newcommand{\hangs}[1]{\hang{\ #1}}
\let\oldtau\tau
\renewcommand{\tau}{\oldtau_1}
\newcommand{\taub}{\oldtau_{\mathrm{B}}}
\newtheorem{proposition}{Proposition}[section]
\newtheorem{conjecture}[proposition]{Conjecture}
\newtheorem{lemma}[proposition]{Lemma}
\newtheorem{theorem}[proposition]{Theorem}
\newtheorem{corollary}[proposition]{Corollary}
\theoremstyle{definition}

\theoremstyle{remark}

\begin{document}
\begin{abstract}
  Erd\H os, Gallai, and Tuza posed the following problem: given an
  $n$-vertex graph $G$, let $\tau(G)$ denote the smallest size of a
  set of edges whose deletion makes $G$ triangle-free, and let
  $\aph(G)$ denote the largest size of a set of edges containing at
  most one edge from each triangle of $G$. Is it always the case that
  $\aph(G) + \tau(G) \leq n^2/4$?

  We have two main results. We first obtain the upper bound $\aph(G) +
  \tau(G) \leq 5n^2/16$, as a partial result towards the Erd\H
  os--Gallai--Tuza conjecture.  We also show that always $\aph(G) \leq
  n^2/2 - m$, where $m$ is the number of edges in $G$; this bound is
  sharp in several notable cases.
\end{abstract}
\maketitle
\section{Introduction}
Given an $n$-vertex graph $G$, say that a set $A \subset E(G)$ is
\emph{triangle-independent} if it contains at most one edge from each
triangle of $G$, and say that $X \subset E(G)$ is a \emph{triangle
  edge cover} if $G \setminus X$ is triangle-free. Throughout this paper,
$\aph(G)$ denotes the maximum size of a triangle-independent set
of edges in $G$, while $\tau(G)$ denotes the minimum size of a triangle
edge cover in $G$.

Erd\H os~\cite{largebip} showed that every $n$-vertex graph $G$ has a
bipartite subgraph with at least $\sizeof{E(G)}/2$ edges, which implies that
$\tau(G) \leq \sizeof{E(G)}/2 \leq n^2/4$. Similarly, if $A$ is
triangle-independent, then the subgraph of $G$ with edge set $A$ is
clearly triangle-free; by Mantel's~Theorem, this implies that $\aph(G)
\leq n^2/4$.

Intuitively, $\aph(G)$ and $\tau(G)$ cannot both be large: if
$\tau(G)$ is close to $n^2/4$, then $\sizeof{E(G)}$ is close to
$n^2/2$, which makes it difficult to find a large triangle-independent
set of edges. Erd\H os, Gallai, and Tuza formalized this intuition
with the following conjecture.
\begin{conjecture}[Erd\H os--Gallai--Tuza \cite{EGT}]\label{coj:EGT}
  For every $n$-vertex graph $G$, $\aph(G) + \tau(G) \leq n^2/4$.
\end{conjecture}
The conjecture is sharp, if true: consider the graphs $K_{n}$ and
$K_{n/2, n/2}$, where $n$ is even. We have $\aph(K_n) = n/2$ and
$\tau(K_n) = {n \choose 2} - n^2/4$, while $\aph(K_{n/2, n/2}) =
n^2/4$ and $\tau(K_{n/2, n/2}) = 0$.  In both cases, $\aph(G) +
\tau(G) = n^2/4$, but a different term dominates in each case. As
observed by Erd\H os, Gallai, and Tuza, the difficulty of the
conjecture lies in the variety of graphs for which the conjecture is
sharp: any proof of the conjecture would need to account for both
$K_n$ and $K_{n/2,n/2}$ without any waste.

Erd\H os, Gallai, and Tuza~\cite{EGT} considered the conjecture on
graphs for which every edge lies in a triangle, and proved that there
is a positive constant $c$ such that $\aph(G) + \tau(G) \leq
\sizeof{E(G)} - c\sizeof{E(G)}^{1/3}$ and $\aph(G) + \tau(G) \leq
\sizeof{E(G)} - c\sizeof{V(G)}^{1/2}$ for such graphs. Aside from the
original paper of Erd\H os, Gallai, and Tuza, no other work appears to
have been done on the conjecture. The conjecture also appears as
Problem~46 in \cite{tuza2001unsolved}, a list of unsolved
combinatorial problems.

In this paper, we present two partial results towards
Conjecture~\ref{coj:EGT}.

In Section~\ref{sec:indbip}, we extend some ideas of Erd\H os,
Faudree, Pach and Spencer~\cite{howbip} in order to obtain the bound
$\aph(G) + \tau(G) \leq 5n^2/16$.  In Section~\ref{sec:aphbound}, we
obtain the bound $\aph(G) \leq n^2/2 - m$, where $m = \sizeof{E(G)}$,
and characterize the graphs for which equality holds. When $n$ is
even, this bound is sharp for both $K_n$ and $K_{n/2,n/2}$, which
makes it an encouraging step towards the Erd\H os--Gallai--Tuza
Conjecture.

\section{Induced Bipartite Subgraphs}\label{sec:indbip}
In this section, we will focus on the relationship between
triangle-free subgraphs of $G$ and bipartite subgraphs of $G$.  The
problem of finding a largest bipartite subgraph of a graph $G$ is
well-studied, and clearly any bipartite subgraph of $G$ is
triangle-free, so we can reasonably hope to apply some of the existing
literature on bipartite subgraphs to our current problem (see
\cite{Bollobas} and \cite{PoljakTuza} for surveys of this literature).

We define some useful notation. For any graph $G$, let $\taub(G)$
denote the smallest size of an edge set $X$ such that $G \setminus X$ is
bipartite, and let $b(G)$ denote the largest size of a vertex set $B$
such that $G[B]$ is bipartite.  Clearly, $\tau(G) \leq \taub(G)$, so
we seek bounds on $\aph(G) + \taub(G)$. When $A \subset E(G)$, we will
abuse notation slightly by identifying $A$ with the spanning subgraph
of $G$ having edge set $A$. This yields notation like $N_A(v)$,
referring to the neighborhood of a vertex $v$ in the spanning subgraph
of $G$ with edge set $A$.

The relationship between $\tau(G)$ and $\taub(G)$ has been studied
before.  Erd\H os \cite{Erdos-biptri} asked which graphs $G$ satisfy
$\tau(G) = \taub(G)$. The question was pursued by Bondy, Shen,
Thomass\'e and Thomassen~\cite{BSTT}, who proved that $\tau(G) =
\taub(G)$ when $\delta(G) \geq 0.85\sizeof{V(G)}$, and later by
Balogh, Keevash, and Sudakov~\cite{Balogh}, who proved that $\tau(G) =
\taub(G)$ when $\delta(G) \geq 0.79\sizeof{V(G)}$.

Erd\H os, Faudree, Pach, and Spencer~\cite{howbip} studied the problem
of finding a largest bipartite subgraph of a triangle-free graph,
using the observation that if $uv$ is an edge of a triangle-free graph,
then $G[N(u) \cup N(v)]$ is an induced bipartite subgraph of
$G$. Here, we use a similar observation:
\begin{lemma}\label{lem:edgebip}
  For any graph $G$, any triangle-independent set $A \subset E(G)$, and any edge $uv \in E(G)$,
  \[ d_A(u) + d_A(v) \leq b(G). \]
\end{lemma}
\begin{proof}
  Since $A$ is triangle-independent, the sets $N_A(u)$ and $N_A(v)$
  are independent and disjoint. Hence $G[N_A(u) \cup N_A(v)]$ is
  bipartite with $d_A(u) + d_A(v)$ vertices.
\end{proof}
\begin{lemma}\label{lem:bip-aph}
  For any graph $G$,
  \[ \aph(G) \leq \frac{nb(G)}{4\hangs.} \]
\end{lemma}
\begin{proof}
  Let $A$ be any triangle-independent subset of $E(G)$. Applying
  Lemma~\ref{lem:edgebip} to all edges in $A$ and summing together
  the resulting inequalities gives
  \[ \sum_{u \in V(G)}d_A(u)^2 = \sum_{uv \in A}[ d_A(u) + d_A(v) ] \leq \sizeof{A}b(G). \]
  By the Cauchy--Schwarz Inequality, we have
  \[ \sum_{uv \in A}d_A(u)^2 \geq \frac{4\sizeof{A}^2}{n\hangs.} \]
  The desired inequality follows.
\end{proof}
\begin{lemma}\label{lem:bip-tau}
  For any graph $G$,
  \[ \taub(G) \leq \frac{n^2}{4} - \frac{b(G)^2}{4\hangs.} \]
\end{lemma}
\begin{proof}
  This is essentially the $\delta = 0$ case of Proposition~2.5 of
  \cite{howbip}.  We sketch a probabilistic proof here. Let $B$ be the
  vertex set of a largest bipartite induced subgraph of $G$. If we
  randomly place the vertices of $V(G) \setminus B$ into the partite sets of
  $B$ and delete all edges within the partite sets, the expected
  number of deleted edges is $\frac{1}{2}\sizeof{E(G) \setminus E(G[B])}$; hence $G$ can be
  made bipartite by deleting at most this many edges.  Thus,
  \[ \taub(G) \leq \frac{1}{2}\sizeof{E(G) \setminus E(G[B])} \leq
  \frac{1}{2}\left[{n \choose 2} - {b(G) \choose 2}\right] \leq
  \frac{n^2}{4} - \frac{b(G)^2}{4\hangs,} \] where the second
  inequality uses the fact that for any vertex set $T \subset V(G)$,
  there are at most ${n \choose 2} - {\sizeof{T} \choose 2}$ edges in
  $E(G) \setminus E(G[T])$.
\end{proof}
\begin{corollary}\label{cor:nearperf}
  For any graph $G$,
  \[ \aph(G) + \tau(G) \leq \aph(G) + \taub(G) \leq \frac{5n^2}{16\hangs.} \]  
\end{corollary}
\begin{proof}
  From Lemma~\ref{lem:bip-aph} and Corollary~\ref{lem:bip-tau} we immediately have
  \[ \aph(G) + \taub(G) \leq \frac{n^2}{4} + \frac{nb(G)}{4} - \frac{b(G)^2}{4\hangs.} \]
  Since the product $x(n-x)$ is maximized when $x = n/2$, this implies
  \[ \aph(G) + \taub(G) \leq \frac{5n^2}{16\hangs,} \]
  as desired.
\end{proof}
While the parameter $\taub(G)$ has been extensively studied, the sum
$\aph(G) + \taub(G)$ has not, to our knowledge, been previously studied.
We have not found any graph $G$ with $\aph(G) + \taub(G) > \sizeof{V(G)}^2/4$,
so we close with the following conjecture, which strengthens Conjecture~\ref{coj:EGT}.
\begin{conjecture}\label{coj:bipcoj}
  For every $n$-vertex graph $G$, $\aph(G) + \taub(G) \leq n^2/4$.
\end{conjecture}
Computer search suggests that Conjecture~\ref{coj:bipcoj} is true for all graphs on
at most $8$ vertices.
\section{Bounding $\aph(G)$}\label{sec:aphbound}
In this section, we will obtain the bound $\aph(G) \leq \lfrac{n^2}{2}
- m$, where $m = \sizeof{E(G)}$.  We first need one quick lemma.
\begin{lemma}\label{lem:vertmax}
  Let $G$ be an $n$-vertex graph, and let $A \subset E(G)$ be
  triangle-independent.  For every edge $uv \in A$, we have $d_A(u)
  \leq n - d_G(v)$.
\end{lemma}
\begin{proof}
  The set $A$ cannot contain any edge $uw$ where $w \in N_G(v)$, since
  then $A$ would contain two edges of the triangle $uvw$. Hence $N_A(u) \subset V(G) \setminus N_G(w)$.
\end{proof}
The \emph{join} of the graphs $G_1, \ldots, G_t$, written $G_1 \join \cdots \join G_t$,
is the graph obtained from the disjoint union $G_1 + \cdots + G_t$ by adding all edges
between vertices from different $G_i$.
\begin{theorem}\label{thm:match}
  For an $n$-vertex graph $G$ with $m$ edges,
  \begin{equation}
    \label{eq:dog}
    \aph(G) \leq \frac{n^2}{2} - m.
  \end{equation}   
  Equality holds if and only if there exist $r_1, \ldots, r_t \geq 1$
  such that $G \iso K_{r_1, r_1} \join \cdots \join K_{r_t, r_t}$.
\end{theorem}
\begin{proof}
  Let $A \subset E(G)$ be triangle-independent, and let $M$ be a
  maximal matching in $A$.  We study the degree sum $\sum_{v \in
    V(G)}d_A(v)$ by splitting it into the sum $\sum_{v \in V(M)}d_A(v)
  + \sum_{v \notin V(M)}d_A(v)$.

  For each $v$ covered by $M$, let $v'$ be its mate in $M$.  Applying
  Lemma~\ref{lem:vertmax} to both endpoints of each edge in $M$, we
  obtain the bound
  \[
    \sum_{v \in V(M)}d_A(v) \leq \sum_{v \in V(M)}[n - d_G(v')] = \sum_{v \in V(M)}[n - d_G(v)].
    \]
  To bound $\sum_{v \notin V(M)}d_A(v)$, we first observe that the vertices not
  covered by $M$ form an independent set in $A$, since any edge
  joining such vertices could be added to obtain a larger matching.

  Now let $v$ be any vertex not covered by $M$.  For each edge $ww'
  \in M$, if $vw \in A$ then $vw' \notin E(G)$, since otherwise $A$
  contains two edges of the triangle $vww'$. Hence $d_A(v) \leq n - 1 -
  d_G(v)$, since each $A$-edge $vw$ is witnessed by a non-$G$-edge $vw'$ where $w' \neq v$.
  Summing this inequality over all uncovered $v$ yields
  \begin{equation}
    \label{eq:cat}
    \sum_{v \notin V(M)}d_A(v) \leq \sum_{v \notin V(M)}[n - 1 - d_G(v)] \leq \sum_{v \notin V(M)}[n - d_G(v)].
  \end{equation}
  Combining this with the bound on $\sum_{v \in V(M)}d_A(V)$ and applying the degree-sum formula for $G$ yields
  \[ \sum_{v \in V(G)}d_A(v) \leq \sum_{v \in V(G)}[n - d_G(v)] = n^2 - 2m. \]
  Applying the degree-sum formula for $A$ completes the proof of the first claim.
  \medskip

  Now we characterize the graphs for which equality holds. First we
  argue that if $G \iso K_{r_1,r_1} \join \cdots \join K_{r_t, r_t}$,
  then $\aph(G) \geq \frac{n^2}{2} - m$.  Let $V_i$ be the subset of
  $V(G)$ containing the vertices of the $i$th graph in this join, and let $A =
  E(G[V_1]) \cup \cdots \cup E(G[V_t])$. The set $A$ is
  triangle-independent: if $uvw$ is a triangle in $G$ with $uv \in A$,
  then $uv \in E(G[V_i])$ for some $i$, so $u$ and $v$ have no common
  neighbors in $G[V_i]$, so that $uw \in A$ implies $vw \notin E(G)$ and
  vice versa.
  Thus,
  \[ \aph(G) \geq \sizeof{A} = \sum_{i=1}^t (r_i)^2 = \sum_{i=1}^t
  \left(r_i + 2{r_i \choose 2}\right) = \frac{n}{2} + 2\sum{r_i
    \choose 2} = \frac{n^2}{2} - m. \] Next we show that these are the
  only graphs for which equality holds.  Let $A$ be a
  triangle-independent subset of $G$ with $\sizeof{A} = \lfrac{n^2}{2}
  - m$. Since equality holds in \eqref{eq:dog}, equality must also hold in
  \eqref{eq:cat} for every maximal matching $M \subset A$.  This is
  only possible if the sum $\sum_{v \notin V(M)}[n - 1 - d_G(v)]$ is
  empty.  Thus, all maximal matchings in $A$ are perfect matchings.

  Let $P_4$ denote the path on four vertices. We claim that if $A$
  contains an induced subgraph isomorphic to $P_4$, then $A$ contains
  a nonperfect maximal matching. Let $v_1,\ldots,v_4$ be the vertices
  of an induced copy of $P_4$ in $A$, written in order, and let $M$ be
  any maximal matching containing the edges $v_1v_2$ and $v_3v_4$; we
  may assume that $M$ is a perfect matching. Now $(M \setminus
  \{v_1v_2,v_3v_4\}) \cup \{v_2v_3\}$ is a nonperfect maximal
  matching, since $v_1v_4 \notin A$.

  Thus, if equality holds in \eqref{eq:dog}, then $A$ is a triangle-free
  graph with a perfect matching and no induced $P_4$. This implies
  that every component of $A$ is a balanced complete bipartite graph.

  Next we claim that if $u$ and $v$ are vertices in different
  components of $A$, then $uv \in E(G)$. Suppose $uv \notin E(G)$, and
  let $G' = G + uv$.  Now $A$ still contains at most one edge from
  any triangle of $G'$: if not, then $A$ contains two edges of $uvz$,
  for some $z \in V(G')$.  Since $uv \notin A$, this implies that $uz
  \in A$ and $vz \in A$, contradicting the hypothesis that $u$ and $v$
  are in different components of $A$. It follows that $\sizeof{A} \leq
  \lfrac{n^2}{2} - \sizeof{E(G')} < \lfrac{n^2}{2} - m$, contradicting
  the assumption that $\sizeof{A} = \lfrac{n^2}{2} - m$.

  We have shown that the vertices of $G$ can be covered by
  vertex-disjoint complete bipartite graphs, and that if $u$ and $v$
  are covered by different graphs, then $uv \in E(G)$. This implies
  that $G \iso K_{r_1,r_1} \join \cdots \join K_{r_t, r_t}$.
\end{proof}
\bibliographystyle{amsplain}\bibliography{sumbib}
\end{document}